\documentclass[11pt]{amsart}
\usepackage{}
\usepackage{amssymb}
\usepackage{xcolor}
\usepackage{mathabx}
\usepackage{shuffle}
\usepackage[OT2,T1]{fontenc}
\DeclareSymbolFont{cyrletters}{OT2}{wncyr}{m}{n}
\DeclareMathSymbol{\Sha}{\mathalpha}{cyrletters}{"58}
\usepackage{amsmath}

\usepackage[top=1.2in, bottom=1.2in, left=1.4in, right=1.4in]{geometry}

\input xy
\xyoption{all}
\usepackage{amsfonts}
\usepackage{mathrsfs}
\usepackage{tikz-cd}

\usepackage{graphicx}
\usepackage{amscd,amsbsy,amsthm}
\usepackage[all]{xy}
\usepackage[colorlinks,plainpages,urlcolor=magenta]{hyperref}%\usepackage{subfigure},backref
\usepackage{verbatim}

\usepackage{enumitem}

\usepackage[normalem]{ulem}

\usepackage{tikz}
\usetikzlibrary{arrows,calc}
\tikzset{
>=stealth',
help lines/.style={dashed, thick},
axis/.style={<->},
important line/.style={thick},
connection/.style={thick, dotted},
}

\newcommand {\Omit}[1]{}

\usepackage{stackengine,scalerel}

\newcount\cols
{\catcode`,=\active\catcode`|=\active
 \gdef\Young(#1){\hbox{$\vcenter
 {\mathcode`,="8000\mathcode`|="8000
  \def,{\global\advance\cols by 1 &}%\fBun
  \def|{\cr
        \multispan{\the\cols}\hrulefill\cr
        &\global\cols=2 }%
  \offinterlineskip\everycr{}\tabskip=0pt
  \dimen0=\ht\strutbox \advance\dimen0 by \dp\strutbox
  \halign
   {\vrule height \ht\strutbox depth \dp\strutbox##
    &&\hbox to \dimen0{\hss$##$\hss}\vrule\cr
    \noalign{\hrule}&\global\cols=2 #1\crcr
    \multispan{\the\cols}\hrulefill\cr%
   }}$}}}

\DeclareFontFamily{U}{rsfs}{%
\skewchar\font127}
\DeclareFontShape{U}{rsfs}{m}{n}{%
<-6>rsfs5<6-8.5>rsfs7<8.5->rsfs10}{}
\DeclareSymbolFont{rsfs}{U}{rsfs}{m}{n}
\DeclareSymbolFontAlphabet
{\mathrsfs}{rsfs}
\DeclareRobustCommand*\rsfs{%
\@fontswitch\relax\mathrsfs}

\newdimen\argwidth
\def\db[#1\db]{
 \setbox0=\hbox{$#1$}\argwidth=\wd0
 \setbox0=\hbox{$\left[\box0\right]$}
  \advance\argwidth by -\wd0
 \left[\kern.3\argwidth\box0 \kern.3\argwidth\right]}

\makeatletter
 
  \@addtoreset{equation}{section}
\makeatother

\title[almost strong approximation for linear algebraic groups]{On almost strong approximation for linear algebraic groups}

\author{YANG CAO AND YIJIN WANG}
\address{Yang CAO\newline  School of Mathematics, Shandong University,\newline  Jinan, Shandong 
Province 250100, China}
\email{yangcao1988@gmail.com}
\address{Yijin WANG\newline  School of Mathematics, Shandong University,\newline  Jinan, Shandong 
Province 250100, China}
\email{202420959@mail.sdu.edu.cn}
\date{\today}
\subjclass[2020]{
}

\theoremstyle{definition}
\newtheorem{defi}{Definition}[section]
\newtheorem{exam}[defi]{Example}
\newtheorem{nota}[defi]{Remark}
\newtheorem{notat}[defi]{Notation}

\newtheorem{Q}[defi]{Question}

\theoremstyle{plain}
\newtheorem{thm}[defi]{Theorem}
\newtheorem{prp}[defi]{Proposition}
\newtheorem{lem}[defi]{Lemma}
\newtheorem{coro}[defi]{Corollary}

\begin{document}

\maketitle

\begin{abstract}
Let $G$ be a connected linear algebraic group over a number field $K$. 
In this article, we study the almost strong approximation property (ASA) of $G$ raised by Rapinchuk and Tralle. Building on Demarche's results on strong approximation with Brauer-Manin obstruction, we introduce a necessary and sufficient condition for (ASA) to hold in terms of the Brauer group of $G$. Using the criteria, we conclude that (ASA) can be completely controlled by the Dirichlet density of the places and the splitting field of $G$, 
which generalizes a result of Rapinchuk and Tralle.
\end{abstract}

\hypersetup{linkcolor=black}
\tableofcontents

\section{Introduction}

\subsection{Almost strong approximation and Dirichlet density}
Let $K$ be a number field. We denote by $\Omega_{K}$ the set of places of $K$ and $\infty_{K}$ the archimedean places of $K$, and $\mathbb{A}_{K}$ the adele ring of $K$. 
Let $S\subset \Omega_{K}$ be a set of places, then we define the adele ring of $S$ and the adele ring off $S$ as:
\[
\mathbb{A}_{K,S}:={\prod}'_{v\in S}K_{v} \ \ \ \text{and}\ \ \ \mathbb{A}^{S}_{K}:={\prod}'_{v\notin S}K_{v},
\]
where the restricted product is taken over $\mathcal{O}_{v}$, the ring of integers of the local field $K_{v}$.

Let $X$ be a smooth geometrically integral variety over $K$. As usual, we define
\[
X(\mathbb{A}_{K,S}):={\prod}'_{v\in S}X(K_{v})\ \ \ \text{and}\ \ \   X(\mathbb{A}^{S}_{K}):={\prod}'_{v\notin S}X(K_{v}),
\]
where the restricted product is taken over $\mathcal{X}(\mathcal{O}_{v})$ for some integral model $\mathcal{X}$ of $X$. See \cite{con12} for further details.

Consider the diagonal embedding $X(K)\to X(\mathbb{A}^{S}_{K})$.
We denote by $\overline{X(K)}^{S}$ the closure of $X(K)$ in $X(\mathbb{A}^{S}_{K})$. 

\begin{defi} Assume that $X(\mathbb{A}_{K})\neq \emptyset$.
We say that $X$ satisfies \textbf{strong approximation} (SA) off $S$ if $\overline{X(K)}^{S}=X(\mathbb{A}^{S}_{K})$.
\end{defi}

\begin{nota} As shown in algebraic number theory, 
the diagonal embedding $K\to \mathbb{A}_{K}$ identifies $K$ with a discrete subset of the adele ring $\mathbb{A}_{K}$. It follows that the image of the diagonal embedding $X(K)\hookrightarrow X(\mathbb{A}_{K})$ is discrete for any quasi-affine variety $X$. 
Hence when we study the (SA) property for quasi-affine varieties (for example, linear algebraic groups), we have to omit a non-empty set of places in $\Omega_{K}$. 
\end{nota}

Now we consider the (SA) property for linear algebraic groups.

\begin{notat}\label{notation-alg-gp}
Let $G$ be a connected linear algebraic group.

 Let $U(G)$ be the unipotent radical of $G$ and $G^{red}:=G/U(G)$ the maximal reductive quotient of $G$. 
 For $G^{red}$, we denote by $Z(G^{red})^0$ its maximal central torus and
$G^{ss}$ its maximal semi-simple subgroup. 
 
 We denote the universal covering of $G^{ss}$ by $G^{sc}$ and $G^{scu}:=G\times_{G^{red}}G^{sc}.$
 By \cite[Thm. 2.4]{PR94}, we have the canonical central isogeny induced by multiplication:
\begin{equation}\label{isogency}
\tau: G^{sc}\times Z(G^{red})^{0}\to G^{red}.
\end{equation}
\end{notat}

When $G$ is semi-simple simply connected, the strong approximation property for 
$G$ has been extensively studied by Shimura(\cite{Shi64}), Kneser(\cite{Kne65}), Platonov(\cite{Pla69}), Prasad(\cite{Pra77}) and others. One of the most important results is the following theorem.

\begin{thm} [Kneser, Platonov] \label{thm-SA}
Let $K$ be a number field and $G$ a semi-simple simply connected algebraic group over $K$.
Let $S$ be a finite non-empty set of places 
such that $G'_{S}:=\prod_{v\in S}G'(K_{v})$ is non-compact for any almost $K$-simple factor $G'$ of $G$. 
Then $G$ satisfies (SA) off $S$.
\end{thm}

On the other hand, when a variety $X$ is not simply connected, Minchev pointed out that $X$ does not satisfy (SA) off any \textbf{finite} set of places (see \cite[Thm. 1]{Mi89}).

To study the behavior of $\overline{G(K)}^{S}$ in $G(\mathbb{A}^{S}_{K})$ in the case that $S\subset \Omega_{K}$ is infinite and $G$ is not simply connected, 
Rapinchuk and Tralle have recently proposed the “almost strong approximation property” for algebraic groups, which is a weaker condition than (SA).

\begin{defi}[\cite{AW25}]
Let $S\subset \Omega_{K}$ be an infinite set of places. We say that $G$ satisfies \textbf{almost strong approximation} (ASA) off $S$ if $[G(\mathbb{A}^{S}_{K}):\overline{G(K)}^{S}]< +\infty$. 
\end{defi}

\begin{exam} [\cite{Ra12}, Proposition 2.1] Let $T$ be a torus over a global field $K$. Assume that $S\subset \Omega_{K}$ is a finite set of places, then 
\[
[T(\mathbb{A}^{S}_{K}):\overline{T(K)}^{S}]=\infty.
\]
Therefore, we require the set of places $S$ to be infinite when we study the (ASA) property for general linear algebraic groups.
\end{exam}

\begin{nota} If $G$ satisfies (ASA) off $S$, then there exists a finite set of places $S'$ such that $G$ satisfies (SA) off $S\cup S'$ (see \cite{AW25}, Definition 2.4).
\end{nota}

\begin{nota}(Proposition \ref{The quotient is abelian})
Let $G$ be a connected linear algebraic group over a number field $K$ 
and $S\supset \infty_K$ an infinite set of places of $K$.
Then the closure $\overline{G(K)}^{S}$ is a normal subgroup of $G(\mathbb{A}^{S}_{K})$ with abelian quotient.
\end{nota}

When $S$ is a certain arithmetic progression (see \cite{AW25}, Definition 1.1) and $G$ is an algebraic torus, Prasad and Rapinchuk have already obtained the (ASA) property for $G$ in (\cite{PR01}). Recently, Rapinchuk and Tralle have established a sufficient condition for the validity of the (ASA) property off $S$ for reductive groups (\cite{AW25}, Theorem 1.3).
Note that the arithmetic progressions always have positive Dirichlet density.

The following result generalizes \cite[Theorem 1.3]{AW25}, demonstrating that the condition ``arithmetic progressions'' can be weakened to ``a set with positive Dirichlet density''.

\begin{thm} \label{thm-(ASA) with inner forms}
 Let $G$ be a connected linear algebraic group over a number field $K$.
 Recall $Z(G^{red})^0$ and $G^{ss}$ from Notation \ref{notation-alg-gp}.
 Let $E$ be the minimal splitting field of $Z(G)^0$ and $M$ the minimal Galois extension of $K$ such that $G^{ss}$ becomes an inner form of a $K$-split group over $M$. Set $L:=EM$.

Let $S\supset \infty_K$ be an infinite set of places of $K$ such that 
the set of places in $S$ that split in $L$ has positive Dirichlet density.
Then $G$ satisfies (ASA) off $S$. 
\end{thm}

%  If $E/\mathbb{Q}$ is Galois, then we can remove the assumption ``$S\supset \infty_K$'' (Remark \ref{outside-infty}).

Theorem \ref{thm-(ASA) with inner forms} is a consequence of the following result.

\begin{thm}\label{thm-mainthm}
Let $G$ be a connected linear algebraic group over a number field $K$. Recall $Z(G^{red})^0$, $G^{red}$, $G^{ss}$, $G^{sc}$ from Notation \ref{notation-alg-gp}.
Let $Q:=\mathrm{Ker(\tau)}$ be the kernel of the central isogeny $\tau: G^{sc}\times Z(G^{red})^{0}\to G^{red}$ and $\hat{Q}$ its Cartier dual.
Let $L/K$ be a Galois extension such that both $Z(G^{red})^{0}$ and $\hat{Q}$ are split over $L$.

Let $S\supset \infty_K$ be an infinite set of places of $K$ such that 
the set of places in $S$ that split in $L$ has positive Dirichlet density.
Then $G$ satisfies (ASA) off $S$. 
\end{thm}

Note that $Z(G^{red})^0$ is split over $L$ if and only if $\hat{G}$ (the character group of $G$) is split over $L$ 
(Example \ref{splitting field and isogeny} (2)).
Moreover, it follows from equation (\ref{eq-density}) that:
the set of places in $S$ that split in $L$ has positive Dirichlet density if and only if $S_{L}$ has positive Dirichlet density in $L$.

\begin{Q}
Under the hypothesis of Theorem \ref{thm-mainthm}, we choose an equivariant smooth compactification $G\subset X$,
a height function $h: X(K)\to \mathbb{R}_{\geq 0}$ (for example, the canonical height) 
and a compact open subset $W\subset G(\mathbb{A_K^S})$. 
What is the asymptotic behavior of
$$N(G,h,W,B):= \frac{\# \{g\in G(K)\cap W|\ h(g)\leq B  \}}{\#\{g\in G(K)|\ h(g)\leq B  \} }  $$
as $B\to \infty$?
\end{Q}

If $G$ is not isogenous to a product of low dimensional subgroups
and $L$ is the minimal Galois extension such that the hypothesis of Theorem \ref{thm-mainthm} is satisfied, 
then we conjecture that
$$N(G,h,W,B) \sim  \frac{C}{log(B)^{d\cdot (1-\delta_L(S_L))}} $$
with $C,d$ are constants.  

We are interested in the case where $W:=\prod_{v\notin S}\mathcal{G}(\mathcal{O}_v)$
with $\mathcal{G}$ an integral model of $G$.
Then $G(K)\cap W$ is exactly the $S$-integral points of $\mathcal{G}$.

\subsection{Almost strong approximation and the Brauer-Manin obstruction}
Let $X$ be a smooth geometrically integral variety over a number field $K$. 
We denote by
\[
\mathrm{Br}(X):=H^2_{\acute{e}t}(X,\mathbb{G}_{\mathrm{m}})
\]
the cohomological Brauer group of $X$. 

Given a local point $P_{v}\in X(K_{v})$, we have the evaluation map $\mathrm{Br}(X)\to \mathrm{Br}(K_{v})$ defined by the pull-back of $P_{v}:\mathrm{Spec}(K_{v})\to X$ on the cohomology groups. 
We denote the pull-back of $b\in \mathrm{Br}(X)$ by $b(P_v)$.

We then have the Brauer-Manin pairing:
$$
\langle-,-\rangle: X(\mathbb{A}_{K})\times\mathrm{Br}(X)\to \mathbb{Q}/\mathbb{Z},\  \left( (P_{v}),b\right)\mapsto \sum_{v\in \Omega_{K}}\mathrm{inv}_{v}b(P_{v}).
$$
where $\mathrm{inv}_{v}:\mathrm{Br}(K_{v})\to \mathbb{Q}/\mathbb{Z}$ is the local invariant map. 
The Brauer-Manin pairing is well defined (see \cite[Proposition 8.2.1]{Po23}).

For any subset $B\subset \mathrm{Br}(X)$, 
we define
\[
X(\mathbb{A}_{K})^{B}:=\{(P_{v})\in X(\mathbb{A}_{K})\ |\ \langle (P_{v}),b\rangle=0, \forall b\in B\}.
\]
The class field theory implies that $X(K)\subset X(\mathbb{A}_{K})^{B}$, and that $X(\mathbb{A}_{K})^{B}$ is closed in $X(\mathbb{A}_{K})$. For these facts and more about the Brauer-Manin pairing, see \cite[Chap. 8]{Po23} and \cite[Chap. 13]{CT21}.

Moreover, the Brauer-Manin pairing induces a canonical continuous map 
$$
a_X: X(\mathbb{A}_{K})\to \mathrm{Hom}(B,\mathbb{Q}/\mathbb{Z}).
$$

Let $G$ be a connected linear algebraic group.
Let $\mathrm{Br}_e(G)$ be the modified algebraic Brauer group, as defined in (\ref{defofBra}), and 
$$\Sha^1(K,G):=\mathrm{Ker}(H^1(K,G)\to \prod_{v\in \Omega_{K}}H^1(K_{v},G))  $$
the Tate-Shafarevich group of $G$.

After a series of works (\cite{CT-Xu}, \cite{HS05}, \cite{HS08}, \cite{H08}, \cite{De09}, \cite{De11}, \cite{BD13}), Demarche established the following result:

\begin{thm} [\cite{De11},Corollary 3.20] \label{thm-Dem11}
Let $G$ be a connected linear algebraic group over a number field $K$ and let $S$ be a finite set of places such that $G^{sc}$ satisfies strong approximation off $S$. Then we have the following exact sequence of groups:
$$\xymatrix{
1 \ar[r] & \overline{G(K)\cdot G^{scu}_{S}\cdot G_{\infty}^+}\ar[r] & G(\mathbb{A}_{K}) \ar[r]^-{a_G} & {\mathrm{Hom}(\mathrm{Br}_{e}(G),\mathbb{Q}/\mathbb{Z})} \ar[r] & {\Sha^1(K,G)} \ar[r] & 1,
} $$
where $G^{scu}_{S}:=\prod_{v\in S}G^{scu}(K_v)$, 
and  $G_{\infty}^+\subset \prod_{v\in \infty_K}G(K_v)$ the neutral connected component.
\end{thm}

We define the $S$-Shafarevich group of the algebraic Brauer group of $G$ as follows:
\begin{equation}\label{defofB_S(G)}
B_{S}(G):=\mathrm{Ker}(\mathrm{Br}_{e}(G)\to \prod_{v\in S}\mathrm{Br}_{e}(G_{K_v})).
\end{equation}

In this article, we demonstrate the following necessary and sufficient condition for (ASA) to hold in terms of the cohomological obstruction:

\begin{thm} \label{thm-mainlem}
Let $G$ be a connected linear algebraic group over a number field $K$ 
and $S\supset \infty_K$ an infinite set of places of $K$. Then $G$ satisfies (ASA) off $S$ if and only if $B_{S}(G)$ is finite.

Moreover, in this case, we have
\[
[G(\mathbb{A}^{S}_{K}):\overline{G(K)}^{S}]\leq |B_{S}(G)|.
\]
\end{thm}

\section{Notations, terminology and preliminary results}

Here are some widely used notations and conventions.

Let $B$ be an abelian group. 
We denote by $B[n]$ the $n$-torsion subgroup of $B$.
We denote by $B^D:=\mathrm{Hom}_{gp}(B,\mathbb{Q}/\mathbb{Z})$ the Pontryagin dual of $B$, equipped with the compact-open topology,
where $B$ is equipped with the discrete topology.

Let $K$ be a field of characteristic 0.
We denote by $\overline{K}$ the algebraic closure of $K$ and
$\Gamma_{K}:=\mathrm{Gal}(\overline{K}/K)$ the absolute Galois group of $K$.

For any bounded below complex $M$ of discrete $\Gamma_{K}$-modules, we denote by $H^i(K,M)$ its Galois cohomology group.
Moreover, for any Galois extension $L/K$ and any complex $M$ of discrete $\mathrm{Gal}(L/K)$-modules,
we set $H^i(L/K,M):=H^i(\mathrm{Gal}(L/K),M)$.

A variety $X$ over $K$ is a separated scheme of finite type over $K$.
Given a field extension $L/K$, we denote by $X_L$ the base change of $X$ to $L$.
In particular $\overline{X}:=X_{\overline{K}}$ denotes base change of $X$ to the algebraic closure of $K$.
If $X$ is integral, we denote by $K[X]^{\times}$ the group of invertible functions on $X$ and $\mathrm{Pic}(X)$ its Picard group.

\medskip

Let $G$ be a linear algebraic group over $K$. 
As usual, its character group is defined as 
$\hat{G}:=\mathrm{Hom}_{gp}(G_{\overline{K}},\mathbb{G}_{m,\overline{K}}),$ which carries a natural Galois action.
If $G$ is of multiplicative type, then $\hat{G}$ is precisely its Cartier dual.
We denote by $Z(G)$ the center of $G$.

Assume $G$ is connected.
We define $\mathrm{Br}_{1}(G):=\mathrm{Ker}(\mathrm{Br}(G)\to \mathrm{Br}(\overline{G}))$, 
called the \textbf{algebraic Brauer group} of $G$. Moreover, we define:
\begin{equation}\label{defofBra}
\mathrm{Br}_{a}(G):=\mathrm{Br}_{1}(G)/\mathrm{Br}(K), \quad \mathrm{Br}_{e}(G):=\mathrm{Ker}(e^{\star}:\mathrm{Br}_{1}(G)\to \mathrm{Br}(K))
\end{equation}
where $e:\mathrm{Spec}(K)\to G$ is the neutral element. Note that $\mathrm{Br}_{a}(G)\cong \mathrm{Br}_{e}(G)$.

An important tool that will be frequently used in this article is the Sansuc's exact sequence, which we now recall for the convenience of the readers.

\begin{thm}[\cite{san81} Proposition 6.10, Corollary 6.11, Theorem 7.2] \label{thm-sansuc}
Let $G$ be a connected linear algebraic group over a field $K$ of characteristic 0. 

(1) Let $X$ be a smooth integral variety and $Y\to X$ be a torsor under $G$. Then we have the following exact sequence:
\[
0\to K[X]^{\times}\to K[Y]^{\times}\to \hat{G}^{\Gamma_K}\to \mathrm{Pic}(X)\to \mathrm{Pic}(Y)\to \mathrm{Pic}(G)\to \mathrm{Br}(X)\to \mathrm{Br}(Y).
\]

(2) Let $1\to H\to G'\to G\to 1$ be an exact sequence of connected linear algebraic groups. 
Then we have the following exact sequence:
\[
0\to \hat{G}^{\Gamma_K}\to \hat{G'}^{\Gamma_K}\to \hat{H}^{\Gamma_K}\to \mathrm{Pic}(G)\to \mathrm{Pic}(G')\to \mathrm{Pic}(H)\to \mathrm{Br}_{e}(G)\to \mathrm{Br}_{e}(G')\to \mathrm{Br}_{e}(H).
\]

(3) Let $1\to \mu\to G'\to G\to1$ be an isogeny of connected linear algebraic groups. Then we have the following exact sequence:
\[
0\to \hat{G}^{\Gamma_K}\to \hat{G'}^{\Gamma_K}\to \hat{\mu}^{\Gamma_K}\to \mathrm{Pic}(G)\to \mathrm{Pic}(G')\to H^1(K,\hat{\mu})\to \mathrm{Br}_{e}(G)\to \mathrm{Br}_{e}(G').
\]
\end{thm}

The following result is an analogue of \cite[Lem. 2.1]{CDX19}. 

\begin{coro}\label{thm-sansuc-cor}
Under the notation above, one has $\mathrm{Br}_e(G)\cong \mathrm{Br}_e(G^{red})$.
\end{coro}

\begin{proof}
We consider the exact sequence:
\[
1\to U(G)\to G\to G^{red}\to 1.
\]
By Sansuc's exact sequence (Theorem \ref{thm-sansuc} (2)), we have the following exact sequence:
\[
\mathrm{Pic}(U(G))\to \mathrm{Br}_{e}(G^{red})\to  \mathrm{Br}_{e}(G)\to \mathrm{Br}_{e}(U(G)).
\]
Since the underlying scheme of a unipotent group is isomorphic to $\mathbb{A}^{n}_{K}$, both $\mathrm{Pic}(U(G))$ and $\mathrm{Br}_e(U(G))$ are $0$. This implies the desired isomorphism $\mathrm{Br}_{e}(G^{red})\cong \mathrm{Br}_{e}(G)$. 
\end{proof}

Let $T\subset G^{red}$ be a maximal torus, and denote by $T^{sc}$ the inverse image of $T$ in $G^{sc}$, which is also a torus.  
We then define the two-term complex of tori and its Cartier dual as follows
\begin{equation}\label{natation-hatC}
C:=[T^{sc}\to T]\ \ \ \text{and}\ \ \    \hat{C}:=[\hat{T}\to \hat{T}^{sc}],
\end{equation}
where $T^{sc}$ and $\hat{T}$ are placed in degree -1. 
The complex $\hat{C}$ can be used to compute the Brauer group of $G$.

\begin{thm}[{\cite[Corollary 7]{BV09}}] \label{thm-Br_{e}G}
Let $G$ be a connected linear algebraic group over a field $K$ of characteristic 0. Then there is a natural isomorphism:
\[
\kappa: H^1(K,\hat{C})\cong \mathrm{Br}_{e}(G).
\]
\end{thm}

 \begin{proof}
 From Corollary \ref{thm-sansuc-cor},
 we have $\mathrm{Br}_e(G)\cong \mathrm{Br}_e(G^{red})$.
Then the result follows from \cite[Corollary 7]{BV09}.
 \end{proof}

Now we consider the central isogeny in (\ref{isogency}):
$$\tau: G^{sc}\times Z(G^{red})^{0}\to G^{red},$$
with finite central kernel $Q=\mathrm{Ker}(\tau)$.
The projection $Q\subset G^{sc}\times Z(G^{red})^{0}\to Z(G^{red})^{0}$ induces a two-term complex
\begin{equation}\label{def-C0}
C_0:=[Q\to Z(G^{red})^{0}]\ \ \ \text{and its Cartier dual}\ \ \  \hat{C_0}:=[\widehat{Z(G^{red})^{0}}\to \hat{Q}] 
\end{equation}
with $Q$ placed in degree $-1$ and $\hat{Q}$ placed in degree 0.

\begin{coro}\label{cor-C0}
Let $G$ be a connected linear algebraic group over a field $K$ of characteristic 0. 
Then $\tau$ induces a quasi-isomorphism $\hat{C}\to \hat{C_0}$ in the derived category of discrete $K$-modules, 
and we have natural isomorphisms:
$$ \mathrm{Br}_e(G)  \cong H^1(K,\hat{C}) \cong H^1(K,\hat{C_0}) . $$
\end{coro}

\begin{proof}
We claim that the isogeny $\tau$ induces an exact sequence:
\begin{equation}\label{ex-seq-isog}
1\to Q\to T^{sc}\times Z(G^{red})^{0}\xrightarrow{\tau_0} T\to 1,
\end{equation}
where $\tau_0:=\tau|_{T^{sc}\times Z(G^{red})^{0}}$.
Indeed, since $Z(G^{red})^{0}$ is a torus, 
the product $T^{sc}\times Z(G^{red})^{0}$  is again a maximal torus of $G^{sc}\times Z(G^{red})^{0}$.
The maximal torus $T$ of $G^{red}$ always contains the center $Z(G^{red})$ (\cite[Prop. 7.6.4 (iii)]{Sp09}), 
hence $ \tau(T^{sc}\times Z(G^{red})^{0}) \subset T$.
Moreover, since the inverse image of a maximal torus $T$ under an isogeny between reductive groups is again a maximal torus (\cite[22.3]{Bo91}), 
we have $\tau^{-1}(T)=T^{sc}\times Z(G^{red})^{0}$.
This implies the sequence (\ref{ex-seq-isog}) is exact.

The Cartier dual of (\ref{ex-seq-isog}) is the exact sequence
\[
0\to \hat{T}\to  \hat{T}^{sc}\oplus \widehat{Z(G^{red})^{0}}   \to \hat{Q}\to 0.
\]
This exact sequence induces a quasi-isomorphism 
$$\hat{C}:=[\hat{T}\longrightarrow \hat{T}^{sc}]\to \hat{C_0}:=[\widehat{Z(G^{red})^{0}}\to \hat{Q}]$$ 
in the derived category of discrete $\Gamma_K$-modules.
We then conclude the following isomorphisms by Theorem \ref{thm-Br_{e}G}:
$$ \mathrm{Br}_e(G)\cong H^1(K, \hat{C}) \cong H^1(K,\hat{C_0}),$$
which completes the proof.
\end{proof}

\medskip

Recall the notion of a splitting field. For a finite Galois extension $L/K$, we say that a discrete $\Gamma_K$-module $M$ \textbf{is split} over $L$ if the induced $\Gamma_L$-action on $M$ is trivial.
We say that a $K$-torus $T$ \textbf{is split} over $L$ if $\hat{T}$ is split over $L$.
In this case, the field $L$ is called a \textbf{splitting field} of $T$.

\begin{exam}\label{splitting field and isogeny}
(1) Let $\phi: T_1\to T_2$ be an isogeny of tori. Then $T_1$ is split over $L$ if and only if $T_2$ is split over $L$.

Indeed, the morphism $\phi$ induces injective homomorphisms:
$$ \hat{\phi}: \hat{T_2}\hookrightarrow \hat{T_1} \ \ \ \text{and}\ \ \ 
\mathrm{Hom}_k(\hat{\phi},\mathbb{Z}):  \mathrm{Hom}_k(\hat{T_1},\mathbb{Z}) \hookrightarrow \mathrm{Hom}_k(\hat{T_2},\mathbb{Z}).$$
Thus, if $\Gamma_L$ acts trivially on $\hat{T_{1}}$, so does it on $\hat{T}_{2}$. Similarly, triviality of the $\Gamma_{L}$-action on $\mathrm{Hom}_{k}(\hat{T}_{2},\mathbb{Z})$ implies triviality on $\mathrm{Hom}_{k}(\hat{T}_{1},\mathbb{Z})$, and the result follows.

(2) Let $G$ be a connected linear algebraic group. 
Then $Z(G^{red})^0$ is split over $L$ if and only if $\hat{G}$ is split over $L$. Indeed, let $G^{tor}:= G^{red}/G^{ss}$ be the maximal quotient torus of $G^{red}$. 
Then we have $\hat{G}\cong \widehat{G^{tor}}$, and the natural map $Z(G^{red})^0\to G^{tor} $ is an isogeny. The result now follows from (1).
\end{exam}

\medskip

Let $K$ be a number field.

For any Galois extension of number fields $L/K$ and a set of places $S\subset \Omega_{K}$, we denote by $S_{L}$ the set of places of $L$ that lie over places in $S$, and $S_{split}$ the subset of $S$ consisting of places that split in $L$.

Throughout this article, the term \textbf{density} refers exclusively to the Dirichlet density. Namely, for any set $S\subset \Omega_{K}$, we define: 
\begin{equation}\label{def-Dirichlet}
\delta_{K}(S):=\lim_{s\to 1^{-}}\frac{\sum_{v\in S}|\mathbb{F}_v|^{-s}}{\sum_{v\in \Omega_{K}}|\mathbb{F}_v|^{-s}}
\end{equation}
provided the limit exists, where $\mathbb{F}_v$ denotes the residue field of $v$. 

We will freely use the following famous Theorem (see \cite[\S VIII.7, Thm. 7.4]{mil11}):
\begin{thm} [Chebotarev density theorem] \label{Chebotarev-density}
Let $L/K$ be a finite Galois extension of number fields. Then the set of primes of $K$ that split completely in $L$ have Dirichlet density $1/[L:K]$.
\end{thm}
 
 Moreover, the following equality holds (see \cite{mil11}, chapter VI, Proposition 3.2 and Corollary 4.6):
\begin{equation}\label{eq-density}
[L:K]\cdot  \delta_{K}(S_{split})=\delta_L(S_L).
\end{equation}

The following notion generalizes the notion of the Tate-Shafarevich group and also the notion (\ref{defofB_S(G)}).
 
\begin{defi}\label{def-S-Shafa}
Let $M$ be a complex of discrete $\Gamma_{K}$-modules, the ($i$-th) $S$-\textbf{Shafarevich group} of $M$ is:
\[
{\Sha}^{i}_{S}(K,M):= \mathrm{Ker}(H^i(K,M)\to \prod_{v\in S}H^i(K_{v},M)).
\]
\end{defi}

Note that our definition is different from that in \cite[\S 1.4]{mil06}.
It is clear that
\begin{equation}\label{sha-direct-sum}
{\Sha}^{i}_{S}(K,M\oplus N)\cong {\Sha}^{i}_{S}(K,M)\oplus {\Sha}^{i}_{S}(K,N).
\end{equation}

\section{Abelian Galois cohomology of reductive groups}

Let $K$ be a number field or a local field of characteristic 0.
Let $G$ be a connected linear algebraic group over $K$.

We follow the Notation \ref{notation-alg-gp}.
Let $T\subset G^{red}$ be a maximal torus, and let $T^{sc}$ denotes its inverse image in $G^{sc}$.
We denote by $C$ and $\hat{C}$ as in (\ref{natation-hatC}).

When $G$ is reductive,
the abelian Galois cohomology of $G$ is defined as follows:
\[
 H^i_{ab}(K,G):=H^i(K,C).
\]
For $i=0,1$, there is a natural \textbf{abelianization morphism}  (see \cite{Bo98} for details)
\[
ab^{i}:H^i(K,G)\to H^i_{ab}(K,G).
\]

The abelian Galois cohomology, and in particular the maximal torus $T$, encodes significant information about the structure of $G$. In what follows, we summarize several key results that will be used later in this article.
\begin{prp}[\cite{Bo98} Proposition 5.1]\label{lem-ab^{0} is surjective}
Let $K$ be a local field and $G$ a connected reductive algebraic group. The morphism
\[
ab^{0}:H^0(K,G)\to H^0_{ab}(K,C)
\]
is surjective with kernel $\rho(G^{sc}(K))$, where $\rho:G^{sc}\to G^{ss}\to G$ is the canonical homomorphism.
\end{prp}

Cyril Demarche has developed the arithmetic duality theorems for two-term complexes of tori to handle the abelian Galois cohomology of reductive groups. 
\begin{prp} [\cite{De09}, Theorem 3.1] \label{thm-duality for torus}
Let $G$ be a connected linear algebraic group over a  local field $K$. For $i=0,1$, the cup-product pairing
\[
H^i(K,C)\times H^{1-i}(K,\hat{C})\to H^2(K,\mathbb{G}_{\mathrm{m}})\hookrightarrow \mathbb{Q}/\mathbb{Z}
\]
 induces a perfect pairing
\[
H^0(K,C)^{\hat{}}\times H^1(K,\hat{C}) \to \mathbb{Q}/\mathbb{Z}
\]
where $\hat{}$ denotes the profinite completion. Hence the right kernel of the pairing
\[
 H^0(K,C)\times H^1(K,\hat{C}) \to \mathbb{Q}/\mathbb{Z}
\]
is trivial.
\end{prp}

\begin{proof} For the first part, see (\cite{De09}, Theorem 3.1). 
We prove now the second part: for any non-zero $b\in H^1(K,\hat{C})$, 
the above perfect pairing implies that $b$ does not vanish on $H^0(K,C)^{\hat{}}$,
and hence $b$ does not vanish on the dense subset $H^0(K,C)$.
\end{proof}

\begin{prp}[\cite{De11}, Lemme 3.13] \label{funtorility of abelian maps}
Let $G$ be a connected reductive group over a local field K. Then the following diagram is commutative up to a sign:
\[
\begin{tikzcd}
{H^0(K,G)} \arrow[d, "ab^{0}"] \arrow[r, "a_{G}"] & \mathrm{Br}_{e}(G)^{D} \arrow[d, "{\kappa^{D},\cong}"] \\
{H^0(K,C)} \arrow[r]                              & {H^1(K,\hat{C})^D}                                      
\end{tikzcd}
\]
where $a_{G}:G(K)\to \mathrm{Br}_{e}(G)^{D}$ is induced by the local Brauer-Manin pairing.
\end{prp}

\begin{coro} \label{The right kernel is trivial}
Let $G$ be a connected linear algebraic group over a local field $K$. Then the right kernel of the Brauer-Manin pairing
\[
 G(K)\times\mathrm{Br}_{e}(G)\to \mathbb{Q}/\mathbb{Z}
\]
is trivial.
\end{coro}

\begin{proof} 
First, assume that $G$ is reductive. 
By Proposition \ref{funtorility of abelian maps}, we have the following commutative diagram (up to a sign):
\[\begin{tikzcd}
    {H^0(K,C)\times H^1(K,\hat{C})} & {\mathbb Q/\mathbb Z} \\
    {G(K)\times \mathrm{Br}_{e}(G)} & {\mathbb Q/\mathbb Z}
    \arrow[from=1-1, to=1-2]
    \arrow["{\kappa,\cong}", shift left=9, from=1-1, to=2-1] 
    \arrow[no head, from=1-2, to=2-2]
    \arrow[shift left, no head, from=1-2, to=2-2]
    \arrow["{ab^{0}}", shift left=9, from=2-1, to=1-1]
    \arrow[from=2-1, to=2-2]
\end{tikzcd}\]
Let $b\in \mathrm{Br}_{e}(G)$ be an element in the right kernel of the lower pairing. 
Then $b=\kappa(b')$ for a unique $b'\in H^1(K,\hat{C})$. 
This element $b'$ also belongs to the right kernel of the upper pairing,
since $ab^{0}$ is surjective by Proposition \ref{lem-ab^{0} is surjective}. 
Hence we conclude $b'=0$ by Proposition \ref{thm-duality for torus}, and therefore $b=0$.

In general, consider the exact sequence:
\[
1\to U(G)\to G\to G^{red}\to 1.
\]
This induces the following exact sequence:
\[
G(K)\to G^{red}(K)\to H^1(K,U(G))=0.
\]
By Corollary \ref{thm-sansuc-cor}, we have a canonical isomorphism $\mathrm{Br}_e(G)\cong \mathrm{Br}_e(G^{red})$. The result then follows from the reductive case and the functorility of the Brauer-Manin pairing (\cite{CT21}, Proposition 13.3.10).
\end{proof}

\medskip

Now we return to the case where $K$ is a number field.
 
Recall the group $B_S(G)$ defined in (\ref{defofB_S(G)}).
Consider the following sequence of topological groups with continuous homomorphisms:
\begin{equation}\label{mainlemlem-eq-sequence}
\xymatrix{
G(\mathbb{A}_{K,S})\ar[r]^-{\phi} & \mathrm{Br}_{e}(G)^{D} \ar[r]^-{\psi} & B_{S}(G)^{D} \to 0,
} \end{equation}
where  $\phi$ is induced by restricting the Brauer-Manin pairing to $G(\mathbb{A}_{K,S})$,
and $\psi$ is the Cartier dual of the inclusion $B_S(G) \subset \mathrm{Br}_e(G)$. 
Therefore, $\psi$ is a continuous surjective homomorphism of profinite groups.

\begin{prp}\label{mainlemlem}
One has $\mathrm{Ker}(\psi)=\overline{\mathrm{Im}(\phi)}$, where $\overline{\{-\}}$ denotes the topological closure.
\end{prp}

\begin{proof}
By the definition of $B_S(G)$, for any $v\in S$,
every element $b\in B_S(G)$ satisfies $b|_{K_v}=0\in \mathrm{Br}_e(G_{K_v})$.
Hence the composition $\psi\circ \phi$ is $0$.
Since $\psi$ is continuous, the preimage $\psi^{-1}(0)$ is closed in $\mathrm{Br}_e(G)^D$, 
which implies 
\[
\mathrm{Ker}(\psi)\supset \overline{\mathrm{Im}(\phi)}.
\]
It remains to prove the reverse inclusion: 
\[
\mathrm{Ker}(\psi)\subset \overline{\mathrm{Im}(\phi)}.
\]

Let $\mathcal{B}$ be the set of all finite subgroups of $\mathrm{Br}_{e}(G)$.
Then $\mathrm{Br}_{e}(G)^{D}$ is isomorphic to the topological inverse limit
\[
\mathrm{Br}_{e}(G)^{D}\cong \underleftarrow{\lim}_{B\in \mathcal{B}}B^D.
\]
Now for each $B\in \mathcal{B}$, the exact sequence:
$$0\to (B/B\cap B_{S}(G))^D\to B^D\to (B\cap B_{S}(G))^D$$
induces an isomorphism of profinite groups
\begin{equation}\label{mainlemlem-eq-invlimit}
\mathrm{Ker}(\psi)\cong \underleftarrow{\lim}_{B\in \mathcal{B}} (B/B\cap B_{S}(G))^D.
\end{equation}

For any non-empty open subset $W\subset \mathrm{Ker}(\psi)$, after possibly replacing $W$ by a smaller open subset,
we may assume that there exist $B\in \mathcal{B}$ and $\theta\in (B/B\cap B_{S}(G))^D$ such that
$$W= p_B^{-1}(\theta),$$
where $p_B: \mathrm{Ker}(\psi)\to (B/B\cap B_{S}(G))^D$ is the projection map induced by (\ref{mainlemlem-eq-invlimit}).

For any $b\in B\setminus (B\cap B_{S}(G))$, 
there exists a place $v\in S$ such that the image of $b$ in $\mathrm{Br}_e(G_{K_v})$ is nonzero. 
By Corollary \ref{The right kernel is trivial}, there exists an element $N'_b\in G(K_{v})$ 
such that the evaluation $b(N'_b)\neq 0$.
Let $N_b$ be the image of $N'_b$ under the canonical inclusion $G(K_v)\hookrightarrow G(\mathbb{A}_{K,S})$. 
By definition, $\phi(N_b)(b)=\langle N_b,b\rangle$ is the Brauer-Manin pairing, and we have:
$$\phi(N_b)(b)=b(N'_b)+\sum_{w\neq v}inv_wb(e)=b(N'_b)\neq 0,$$
where $e\in G$ is the neutral element.

Now we define a map
\[
\phi_b:B/B\cap B_{S}(G)\to \mathbb{Q}/\mathbb{Z},\quad\beta\mapsto \langle N_b,\beta \rangle
\]
which is well-defined by the definition of $B_{S}(G)$, and $\phi_b(b)\neq 0$.

Let $C\subset \mathrm{Hom}(B/B\cap B_{S}(G),\mathbb{Q}/\mathbb{Z})$ be the subgroup generated by all the maps $\phi_b$. Then the canonical pairing
\[
B/B\cap B_{S}(G)\times C\to \mathbb{Q}/\mathbb{Z}
\]
has trivial left kernel. Hence $C=\mathrm{Hom}(B/B\cap B_{S}(G),\mathbb{Q}/\mathbb{Z})$.

We now consider the homomorphism $\theta: B/B\cap B_{S}(G)\to \mathbb{Q}/\mathbb{Z}$ introduced above.
Then $\theta$ can be written as a sum of $\phi_b$, i.e., there exist integers $n_b$ such that 
$$\theta=\sum_bn_b\phi_b\in \mathrm{Hom}(B/B\cap B_{S}(G),\mathbb{Q}/\mathbb{Z}).$$
Let $N:=\prod_b N_b^{n_b}\in G(\mathbb{A}_{K,S})$, where the product is taken in some fixed order.
Since $\phi$ is a homomorphism, one has $p_B(\phi(N))=\theta$ and $\phi(N)\in W$, which completes the proof of the proposition.
\end{proof}

\section{The proof of the Theorem \ref{thm-mainlem}}

Let $K$ be a number field, $G$ a connected linear algebraic group over $K$, and $S\supset \infty_K$ an infinite set of places.
In this section, we give a necessary and sufficient condition for (ASA) to hold 
in terms of the group $B_{S}(G)$ introduced before (Theorem \ref{thm-mainlem}).

Let $H_{i},i=1,...,n$ be the almost simple factors of $G^{sc}$. 
Then there exist places $v_{i}\in S$ such that $H_{i}$ is isotropic over $K_{v_i}$ (\cite[Thm. 6.7]{PR94}). 
Let $S_0:=\{v_{1},...,v_{n}\}$, it follows that $G^{sc}$ satisfies (SA) off $S_{0}$ by Theorem \ref{thm-SA}.
Hence the condition in Theorem \ref{thm-Dem11} is guaranteed. 

Recall our notations in Theorem \ref{thm-Dem11}:
$G_{S_0}:=\prod_{v\in S_0}G(K_v)$ and $G_{\infty}^+\subset \prod_{v\in \infty_K}G(K_v)$ denotes the neutral connected component.
By Theorem \ref{thm-Dem11}, we have the following commutative diagram with exact rows 
\[
\begin{tikzcd}
  &    & G(\mathbb{A}_{K,S}) \arrow[d, "i_{S}"] \arrow[rd, "\phi"] &    \\
1 \arrow[r] & \overline{G(K)\cdot G_{S_{0}}^{scu}\cdot G_{\infty}^+} \arrow[r] \arrow[d] 
& G(\mathbb{A}_{K}) \arrow[r, "a_{G}"] \arrow[d, "p_{S}"]    
& \mathrm{Br}_{e}(G)^{D} \arrow[r,"\varphi"]  & {{\Sha^1(K,G)}} \arrow[r]  & 1 \\
  & \overline{G(K)}^{S} \arrow[r]       & G(\mathbb{A}^{S}_{K})    &    &  &  
\end{tikzcd}
\]
where  $p_S$ is the projection, $i_S$ is the inclusion, and $\phi:=a_{G}\circ i_{S}$. 
The map $\phi$ is induced by the Brauer-Manin pairing and coincides with the one defined in (\ref{mainlemlem-eq-sequence}). 

Since $\varphi$ is continuous and $\varphi\circ a_G=0$, we have $\overline{\mathrm{Im}(\phi)}\subset \mathrm{Ker}(\varphi)$.

\begin{prp} \label{The quotient is abelian}
Under the notations and hypothesis above, 
the closure $\overline{G(K)}^{S}$ is a normal subgroup of $G(\mathbb{A}^{S}_{K})$, 
the quotient $G(\mathbb{A}^{S}_{K})/\overline{G(K)}^{S}$ is abelian, and we have a canonical isomorphism
$$
    G(\mathbb{A}^{S}_{K})/\overline{G(K)}^{S}\to   \mathrm{Ker}(\varphi)/\overline{\mathrm{Im}(\phi)}.
$$
\end{prp}

\begin{proof}
Since $G(\mathbb{A}_{K})\cong G(\mathbb{A}_{K,S})\times G(\mathbb{A}_{K}^S)$, 
the projection $p_S$ is surjective and we have
$$p_S^{-1}\overline{(G(K))}^S=\overline{G(K)\cdot G(\mathbb{A}_{K,S})},$$
which is a subgroup of $G(\mathbb{A}_{K})$.
Since $S_0\cup \infty_K\subset S$, we have
\begin{equation}\label{assume-infty-inclusion-1}
p_S^{-1}\overline{(G(K))}^S\supset  \overline{G(K)\cdot G_{S_{0}}^{scu}\cdot G_{\infty}^+} .
\end{equation}
By Theorem \ref{thm-Dem11}, the right hand side is a normal subgroup of $G(\mathbb{A}_{K})$ with abelian quotient. Hence, the same property holds for the left hand side.
Therefore, the subgroup $\overline{G(K)}^{S}\subset G(\mathbb{A}^{S}_{K})$ is normal with abelian quotient.
Moreover,
\[
G(\mathbb{A}^{S}_{K})/\overline{G(K)}^{S} \cong G(\mathbb{A}_K)/p^{-1}\overline{(G(K))}^S \cong
G(\mathbb{A}_K)/\overline{G(K)\cdot G(\mathbb{A}_{K,S})} \cong \mathrm{Ker}(\varphi)/\overline{\mathrm{Im}(\phi)},
\]
and we conclude the result.
\end{proof}

\begin{proof}[Proof of Theorem \ref{thm-mainlem}]
Recall the map $\psi$ defined in (\ref{mainlemlem-eq-sequence}).
We claim that we have the following isomorphisms:
$$ G(\mathbb{A}^{S}_{K})/\overline{G(K)}^{S}\cong    \mathrm{Ker}(\varphi)/\overline{\mathrm{Im}(\phi)}\xrightarrow{\iota}
 \mathrm{Br}_e(G)^D/\overline{\mathrm{Im}(\phi)}\cong \mathrm{Br}_e(G)^D/\mathrm{Ker}(\psi)\cong B_S(G)^D,
$$
except that $\iota$ is an injective homomorphism with $\mathrm{Coker}(\iota)\cong \Sha^1(K,G)$, which is a finite group. 
This follows from Proposition \ref{The quotient is abelian}, the definition of $\varphi$, 
Proposition \ref{mainlemlem} and the definition of $\psi$, respectively.
Therefore, $G$ has (ASA) off $S$ if and only if  $B_S(G)$ is finite.
Moreover, in this case,  we have
$$
[G(\mathbb{A}^{S}_{K}):\overline{G(K)}^S]\leq  |B_{S}(G)^D|= |B_{S}(G)|,
$$
which completes the proof. 
\end{proof}

\begin{nota}
In the proof of Theorem \ref{thm-mainlem}, the hypothesis $S\supset \infty_K$ is only used to confirm (\ref{assume-infty-inclusion-1}).
Let $S$ be an infinite set of places of $K$.
We denote $G_{\infty\setminus S}^+\subset \prod_{v\in \infty_K,v\notin S}G(K_v)$ the neutral connected component
and $\overline{G(K)}^{+,S}$ the closure of $G(K)\cdot G_{\infty\setminus S}^+$ in $G(\mathbb{A}^{S}_{K})$.
Actually, the proof of Theorem \ref{thm-mainlem} shows:

\emph{(i) the $\overline{G(K)}^{+,S}$ is a normal subgroup of $G(\mathbb{A}^{S}_{K})$ with abelian quotient;}

\emph{(ii) the index $[G(\mathbb{A}^{S}_{K}): \overline{G(K)}^{+,S}]$ is finite if and only if $B_{S}(G)$ is finite.}
\end{nota}

\medskip

In the following, we will provide a necessary and sufficient condition for the (ASA) property of quasi-split tori
in Proposition \ref{max-split-field-prop}.
To provide this condition, we need the notion of the maximal abelian extension that splits over $S$.

Let $K$ be a number field, $S$ a set of places of $K$ and $\bar{K}$ a fixed algebraic closure of $K$.
We consider the set $\mathcal{E}_S$ of all finite abelian extensions $E/K$ inside $\overline{K}$ 
such that every $v\in S$ splits completely in $E$.
The following lemma shows that the maximal element of $\mathcal{E}_S$ exists and is unique.

\begin{lem}\label{max-split-field-lem1}
There exists a unique maximal element $E_S\in \mathcal{E}_S$, i.e., we have $E\subset E_S$  for any $E\in \mathcal{E}_S$.
\end{lem}

\begin{proof}
For any Galois extension $E/K$, the absolute Galois group $\Gamma_E$ is a normal closed subgroup of $\Gamma_K$.
 For any $v\in S$, the statement that $v$  splits completely in $E$ holds if and only if 
 we have the inclusion $K\subset E\subset K_v$.
This holds if and only if the image of $\Gamma_{K_v}\to \Gamma_K$ is contained in $\Gamma_E$.
Thus $E\in \mathcal{E}_S$ if and only if 
$$\mathrm{Im}(\Gamma_{K_v}\to \Gamma_K) \subset \Gamma_E$$
for any $v\in S$.

Let $E_S:=\prod_{E\in \mathcal{E}_S}E\subset \overline{K}$ 
be the subfield of $\overline{K}$ that generated by all $E\in \mathcal{E}_S$.
Then $E_S/K$ is still an abelian extension and
$$\mathrm{Im}(\Gamma_{K_v}\to \Gamma_K)\subset  \bigcap_{E\in \mathcal{E}_S}\Gamma_E=\Gamma_{E_S} $$
for any $v\in S$. 
Hence $E_S\in \mathcal{E}_S$ and we conclude the result.
\end{proof}

\begin{defi}\label{max-split-field-def}
Let $S$ be a set of places of $K$
and $E_S$ the unique maximal element in Lemma \ref{max-split-field-lem1}.
We call the extension $E_S/K$ the \textbf{maximal abelian extension that splits over $S$}.
\end{defi}

Recall the inflation-restriction exact sequence
$$0\to H^1(E_S/K,\mathbb{Q}/\mathbb{Z}) \to H^1(K,\mathbb{Q}/\mathbb{Z}) \to H^1(E_S,\mathbb{Q}/\mathbb{Z}) . $$
This allows us to identify $H^1(E_S/K,\mathbb{Q}/\mathbb{Z})$ as a subgroup of $H^1(K,\mathbb{Q}/\mathbb{Z})$ 
consisting of those elements $\alpha\in H^1(K,\mathbb{Q}/\mathbb{Z})$ such that $\alpha|_{E_S}=0$.

\begin{lem}\label{max-split-field-lem2}
We have $\Sha^1_S(K,\mathbb{Q}/\mathbb{Z}) =H^1(E_S/K,\mathbb{Q}/\mathbb{Z})$ as subgroups of $ H^1(K,\mathbb{Q}/\mathbb{Z})$,
where $E_S/K$ is the maximal abelian extension that splits over $S$.
\end{lem}

\begin{proof} 
 For any $v\in S$, we have the inclusion $K\subset E_S\subset K_v$, because $v$  splits completely in $E_S$.
Therefore, any $\alpha\in H^1(E_S/K,\mathbb{Q}/\mathbb{Z})$ satisfies $\alpha|_{K_v}=(\alpha|_{E_S})|_{K_v}=0|_{K_v}=0$,
which implies $\alpha\in \Sha^1_S(K,\mathbb{Q}/\mathbb{Z})$.

On the other hand, the following isomorphisms are well-known:
$$
H^1(K,\mathbb{Q}/\mathbb{Z})\cong \mathrm{Hom}_{cont}(\Gamma_{K},\mathbb{Q}/\mathbb{Z})\cong \mathrm{Hom}_{cont}(\Gamma^{ab}_{K},\mathbb{Q}/\mathbb{Z}).
$$
Any $\alpha\in \Sha^1_S(K,\mathbb{Q}/\mathbb{Z})$ corresponds to a continuous homomorphism $\phi_{\alpha}: \Gamma_K\to \mathbb{Q}/\mathbb{Z}$ 
such that $\phi_{\alpha}(\Gamma_{K_v})=0$ for all $v\in S$.
Set $E_{\alpha}:=\overline{K}^{\mathrm{Ker}(\phi_{\alpha})} $.
Then $E_{\alpha}/K $ is a finite abelian extension whose Galois group is $\mathrm{Im}(\phi_{\alpha})$, 
and all $v\in S$ split completely in $E_{\alpha} $.
Since $E_S$ is the unique maximal element in $\mathcal{E}_S$, we have $E_{\alpha}\subset E_S$ (Lemma \ref{max-split-field-lem1}).
Hence we have $\alpha|_{E_S}=(\alpha |_{E_{\alpha}})|_{E_S}=0|_{E_S}=0$, 
which implies $\alpha\in H^1(E_S/K,\mathbb{Q}/\mathbb{Z})$. 
\end{proof}

\begin{prp}\label{max-split-field-prop}
Let $S\supset \infty_K$ be an infinite set of places, $L/K$ a finite field extension and $E_{S_L}/L$ the maximal abelian extension that splits over $S_L$.
Then the Weil restriction $T:=R_{L/K}(\mathbb{G}_m)$ satisfies (ASA) off $S$ if and only if $[E_{S_L}:L]$ is finite.
\end{prp}

\begin{proof}
For any field extension $F/K$, we have the canonical isomorphisms 
$$H^1(F\otimes_KL,\mathbb{Q}/\mathbb{Z})\cong H^2(F\otimes_KL,\mathbb{Z})\cong 
H^2(F,R_{L/K}\mathbb{Z}) \cong \mathrm{Br}_e(T_F),$$
where $ H^i(F\otimes_KL,-):=\oplus_jH^i(F_j,-)$ if $F\otimes_KL=\prod_jF_j$.
Apply to the case $F=K$ and $F=K_v$ for all $v\in S$, then
$\Sha^1_{S_L}(L,\mathbb{Q}/\mathbb{Z})\cong B_S(T).$
By Lemma \ref{max-split-field-lem2}, we have
$$Hom(\mathrm{Gal}(E_{S_L}/L),\mathbb{Q}/\mathbb{Z})\cong  H^1(E_{S_L}/L,\mathbb{Q}/\mathbb{Z})\cong 
\Sha^1_{S_L}(L,\mathbb{Q}/\mathbb{Z})\cong B_S(T).$$
Therefore, the finiteness of $B_S(T)$ is equivalent to the finiteness of $[E_{S_L}:L]$.
The result follows from Theorem \ref{thm-mainlem}.
\end{proof}

\section{The proof of Theorem \ref{thm-mainthm} and Theorem \ref{thm-(ASA) with inner forms}}

In this section, let $K$ be a number field with absolute Galois group $\Gamma_K$ and let $S\subset \Omega_K$ be a set of places.

% Now we explain the proof of Theorem \ref{thm-mainthm}. 
To study (ASA) off $S$ for a connected linear algebraic group $G$,
Theorem \ref{thm-mainlem} shows that it suffices to examine the finiteness of $B_S(G)$ (see (\ref{defofB_S(G)}) for the definition).
This is related to the Galois cohomology of a certain two-term complex (Corollary \ref{cor-C0}).
By this method, we establish Theorem \ref{thm-mainthm1}, which is the key result of this section.
Then Theorem \ref{thm-mainthm} and Theorem \ref{thm-(ASA) with inner forms} follow from Theorem \ref{thm-mainthm1}.

Recall the notion of $S$-Shafarevich group (Definition \ref{def-S-Shafa}).

\begin{lem}\label{sha-Z-Q/Z-Z/n}
We have $\Sha^1_{S}(K,\mathbb{Q}/\mathbb{Z})=\Sha^2_{S}(K,\mathbb{Z})$ and 
$\Sha^1_{S}(K,\mathbb{Z}/n)=\Sha^2_{S}(K,\mathbb{Z})[n]$ for any nonzero integer $n$.
\end{lem}

\begin{proof} 
The exact sequence $0\to \mathbb{Z}\to \mathbb{Q}\to \mathbb{Q}/\mathbb{Z}\to 0$ induces a natural isomorphism
$H^1(F,\mathbb{Q}/\mathbb{Z})\cong H^2(F,\mathbb{Z})$ for any field extension $F/K$.
Applying this isomorphism to $F=K$ and $F=K_v$ for all $v\in S$, 
we conclude that $\Sha^1_{S}(K,\mathbb{Q}/\mathbb{Z})=\Sha^2_{S}(K,\mathbb{Z})$.

It is well known that $H^1(K,\mathbb{Z})=0$ for any field $F$.
Then the canonical exact sequence $0\to \mathbb{Z}\to \mathbb{Z}\to \mathbb{Z}/n\to 0$ induces a natural exact sequence
$$0\to H^1(F,\mathbb{Z}/n)\to H^2(F,\mathbb{Z})\xrightarrow{\times n} H^2(F,\mathbb{Z}) .$$
Applying this exact sequence to $F=K$ and $F=K_v$ for all $v\in S$,
we obtain that $\Sha^1_{S}(K,\mathbb{Z}/n)=\Sha^2_{S}(K,\mathbb{Z})[n].$
\end{proof}

By convention, 
a cochain complex $M$ of $\Gamma_{K}$-modules written as
$$
M=[\cdots \to M_{-1}\to M_0\to \cdots ],
$$
is understood to have its component module $M_{i}$ in degree $i$.

\begin{prp} \label{Going over L}
Let $L/K$ be a finite Galois extension. 
Let  $M=[M_{-1}\to M_{0}]$ be a two-term complex of finitely generated $\Gamma_{K}$-modules with $M_{-1}$ torsion-free.

(1) If $\Sha^1_{S_{L}}(L,M)$ is finite, then $\Sha^1_{S}(K,M)$ is finite.

(2) If $\Sha^1_{S_{L}}(L,\mathbb{Q}/\mathbb{Z})$ is finite and $M_0,M_{-1}$ are split over $L$, then $\Sha^1_{S}(K,M)$ is finite.
\end{prp}

\begin{proof} 
Consider the canonical distinguished triangle:
\begin{equation}\label{disting-trian}
 [0\to M_{0}]\to[M_{-1}\to M_{0}]\to[M_{-1}\to 0]\to +1.
\end{equation}
This distinguished triangle yields the following long exact sequence for any field extension $F/K$:
\begin{equation}\label{exact-two-complex}
\cdots \to H^i(F,M_{-1})\to H^i(F,M_{0})\to H^i(F,M)\to H^{i+1}(F,M_{-1})\to \cdots.   
\end{equation}

We will now prove (1).

Consider the restriction map in Galois cohomology:
\[
 \mathrm{res}_{L/K}: H^1(K,M)\to H^1(L,M)
\]
We claim that $\mathrm{Ker}(\mathrm{res}_{L/K})$ is finite.
To see this, consider the Hochchild-Serre spectral sequence for $L/K$ and $M$:
$$
E_2^{p,q}:=H^p(\mathrm{Gal}(L/K),H^q(L,M))\Rightarrow H^{p+q}(K,M).
$$
Since $H^i(L,M)=0$ for all $i\leq -2$, the spectral sequence yields a natural exact sequence:
$$
H^2(L/K,H^{-1}(L,M)) \to \mathrm{Ker}(\mathrm{res}_{L/K}) \to  H^1(L/K,H^0(L,M)).
$$
Since $M_0$, $M_{-1}$ are finitely generated and $M_{-1}$ is torsion-free,
the groups $H^0(L,M_0)$, $H^0(L,M_{-1})$ and $H^1(L,M_{-1})$ are finitely generated.
By the long exact sequence (\ref{exact-two-complex}), the groups $H^{-1}(L,M)$ and $H^0(L,M)$ are both finitely generated. 
By standard cohomology theory, the groups $H^2(L/K,H^{-1}(L,M))$  and $H^1(L/K,H^0(L,M))$ are both finite.
Hence $\mathrm{Ker}(\mathrm{res}_{L/K})$ is finite.

On the other hand, we have the following commutative diagram:
\[
\begin{tikzcd}[column sep=large]
{\mathrm{Ker}(\mathrm{res}_{L/K})} \arrow[r] & {H^1(K,M)} \arrow[d] \arrow[r, "\mathrm{res}_{L/K}"]   & {H^1(L,M)} \arrow[d]   \\
& {{\prod}_{v\in S}H^1(K_{v},M)} \arrow[r, "\prod \mathrm{res}_{L_{w}/K_{v}}"] & {{\prod}_{w\in S_{L}}H^1(L_{w},M)}
\end{tikzcd}
\]
A diagram chasing yields the following inclusion:
$$
\mathrm{Ker}(\Sha^{1}_{S}(K,M)\to \Sha^1_{S_{L}}(L,M)) \subset \mathrm{Im}(\mathrm{Ker}(\mathrm{res}_{L/K}))
$$
as subgroups of $H^1(K,M)$. Hence the finiteness of $\Sha^1_{S_{L}}(L,M)$ implies the finiteness of $\Sha^1_{S}(K,M)$.
This completes the proof of (1).

Moreover, the above arguments imply that
\begin{equation}\label{Hochchild-counting}
|\Sha^1_{S}(K,M)|\leq |\Sha^1_{S_{L}}(L,M) | \cdot |H^2(L/K,H^{-1}(L,M)) | \cdot | H^1(L/K,H^0(L,M))|.
\end{equation}

We now proceed to prove (2).

By Lemma \ref{sha-Z-Q/Z-Z/n} and the hypothesis, the groups
$\Sha^2_{S_L}(L,\mathbb{Z})$ and $\Sha^1_{S_L}(L,\mathbb{Z}/n)$ are finite for every positive integer $n$.
It is clear that $\Sha^1_{S_L}(L,\mathbb{Z})=0$.
Since $M_0,M_{-1}$ are split over $L$ and $M_{-1}$ is torsion-free,
the equality (\ref{sha-direct-sum}) implies that the groups
$$\Sha^2_{S_L}(L,M_{-1})\ \ \ \text{and}\ \ \ \Sha^1_{S_L}(L,M_0)  $$
are both finite.

On the other hand,
since $M_{-1}$ is torsion-free and split over $L$, we have:
$$H^1(L,M_{-1})=0\ \ \ \text{and}\ \ \ H^1(L_v,M_{-1})=0 $$
for all place $v$ of $L$.
The long exact sequence (\ref{exact-two-complex}) yields a commutative diagram with exact rows
$$
\begin{tikzcd}
{0}\arrow[r]& {H^1(L,M_0)} \arrow[r] \arrow[d] & {H^1(L,M)} \arrow[d] \arrow[r]  
& {H^2(L,M_{-1})} \arrow[d]               \\
{0}\arrow[r] & {{\prod}_{v\in S_L}H^1(L_{v},M_0)} \arrow[r]  & {{\prod}_{v\in S_L}H^1(L_{v},M)} \arrow[r]
& {{\prod}_{v\in S_{L}}H^2(L_v,M_{-1})}.
\end{tikzcd}
$$
A diagram chasing gives an exact sequence:
$$ 0\to \Sha^1_{S_L}(L,M_0) \to \Sha^1_{S_L}(L,M) \to  \Sha^2_{S_L}(L,M_{-1}).$$
The finiteness of $\Sha^1_{S_L}(L,M_{0})$ and $\Sha^2_{S_{L}}(L,M_{-1})$ implies that $\Sha^1_{S_L}(L,M)$ is finite, and the result follows from statement (1).
\end{proof}

Recall the notion of the maximal abelian extension that splits over $S$ (Definition \ref{max-split-field-def}).
The following theorem is a general version of our Theorem \ref{thm-mainthm}.

\begin{thm}\label{thm-mainthm1}
Let $G$ be a connected linear algebraic group over a number field $K$. Recall $Z(G^{red})^0$, $G^{red}$, $G^{ss}$, $G^{sc}$ from Notation \ref{notation-alg-gp}.
Let $Q:=\mathrm{Ker(\tau)}$ be the kernel of the central isogeny $\tau: G^{sc}\times Z(G^{red})^{0}\to G^{red}$ and $\hat{Q}$ its Cartier dual.
Let $L/K$ be a Galois extension such that both $Z(G^{red})^{0}$ and $\hat{Q}$ are split over $L$.

Let $S\supset \infty_K$ be an infinite set of places of $K$
and $E_{S_L}/L$ the maximal abelian extension that splits over $S_L$. 
If $[E_{S_L}:L]$ is finite, then $G$ satisfies (ASA) off $S$. 
\end{thm}

\begin{proof}
Since $[E_{S_L}:L]$ is finite, $H^1(E_{S_L}/L,\mathbb{Q}/\mathbb{Z})$ is finite.
Applying Lemma \ref{max-split-field-lem2} to $L$, we obtain that $\Sha^1_{S_L}(L,\mathbb{Q}/\mathbb{Z})$ is finite.

Recall the notations $\hat{C}$ in (\ref{natation-hatC}) and $\hat{C_0}$ in (\ref{def-C0}).
By Theorem \ref{thm-Br_{e}G} and Corollary \ref{cor-C0}, we have the following isomorphisms:
$$\mathrm{Br}_e(G) \cong H^1(K,\hat{C}) \cong H^1(K,\hat{C_0})\ \ \  \text{and}\ \ \  
\mathrm{Br}_e(G_{K_v})\cong H^1(K_v, \hat{C}) \cong H^1(K_v, \hat{C_0})$$
for every place $v$ of $K$.
Therefore
\begin{equation}\label{result-C0}
B_{S}(G)\cong \Sha^1_{S}(K,\hat{C})\cong \Sha^1_{S}(K,\hat{C_0}).
\end{equation}

The hypothesis of Theorem \ref{thm-mainthm1} implies that $\hat{C_0}$ 
satisfies the condition of Proposition \ref{Going over L} (2) that is imposed on $M$, 
which implies that $\Sha^1_{S}(K,\hat{C_0})$ is finite.
By (\ref{result-C0}), the group $B_{S}(G)$ is finite, and we conclude $G$ satisfies (ASA) off $S$ (Theorem \ref{thm-mainlem}).
\end{proof}

% Recall the definition of Dirichlet density $\delta$ in (\ref{def-Dirichlet}).

\begin{proof}[Proof of Theorem \ref{thm-mainthm}] 
Let $E_{S_L}/L$ be the maximal abelian extension that splits over $S_L$.
By Theorem \ref{thm-mainthm1}, it suffices to show $[E_{S_L}:L]$ is finite.

By hypothesis, the set of places in $S$ that split in $L$ has positive Dirichlet density, 
therefore, the set $S_L$ also has positive Dirichlet density in $L$ by (\ref{eq-density}). 
Since all places in $S_L$ split in $E_{S_L}$,
the Chebotarev density theorem (see Theorem \ref{Chebotarev-density}) implies that 
\begin{equation}\label{density-extension}
 \delta_{L}(S_L)\leq 1/[E_S:L]   
\end{equation}
Therefore, $[E_{S_L}:L]$ is finite and the result follows.
\end{proof}

\begin{coro}\label{cor-mainthm}
Let $K$ be a number field, $L/K$ a finite Galois extension and $S\supset \infty_{K}$ an infinite set of places of $K$
such that the places of $S$ that split completely in $L$ has positive Dirichlet density.

(1) Let $T$ be a torus over $K$ that splits over $L$. Then $T$ satisfies (ASA) off $S$.

(2) Let $G$ be a connected semi-simple algebraic group over $K$ such that $\mathrm{Pic}(\overline{G})$ is split over $L$.
Then $G$ satisfies (ASA) off $S$.
\end{coro}

\begin{proof}
Statement (1) follows directly from Theorem \ref{thm-mainthm}.  For statement (2), let $\tau^{sc}: G^{sc}\to G$ be the universal covering.
Then $\widehat{G^{sc}}=0$, $\mathrm{Pic}(\overline{G^{sc}})=0$ 
and the Sansuc's exact sequence (Theorem \ref{thm-sansuc} (3)) implies
\begin{equation}
\mathrm{Pic}(\overline{G})\cong \widehat{\mathrm{Ker}(\tau^{sc})}\cong \hat{Q},
\end{equation}
which is split over $L$ by hypothesis. The result follows from Theorem \ref{thm-mainthm} again.
\end{proof}

Recall the notion of $K$-forms.
Let $G_1$ be a connected linear algebraic group.
We say a linear algebraic group $G_2$ is a \textbf{$K$-form} of $G_1$ 
if $\overline{G_1}\cong \overline{G_2}$ as $\overline{K}$-groups.
We say $G_2$ is an \textbf{inner $K$-form} of $G_1$ if there exist 
an $\overline{K}$-isomorphism $\theta: \overline{G_1}\to \overline{G_2} $ and a map $\iota: \Gamma_K\to G_1(\overline{K})$, 
such that for every $\sigma\in \Gamma_K$, the following equality holds: 
$$\rho_{\iota(\sigma)}=\theta^{-1}\circ \sigma(\theta),$$ 
where $\sigma(\theta):= \sigma|_{G_2}\circ \theta \circ \sigma^{-1}|_{G_1} $ and
$\rho_{\iota(\sigma)}: \overline{G_1}\to \overline{G_1} $ is the conjugation induced by $\iota(\sigma)$.

\begin{lem}\label{lem-inner-form}
Let $G_1$ be a semi-simple algebraic group and $G_2$ an inner form of $G_1$. 
Then $G_2^{sc}$ is an inner form of $G_1^{sc}$.
\end{lem}

\begin{proof}
We continue to use $\theta$ and $\iota$ as above. 
Since $G_1^{sc}(\overline{K})\to G_1(\overline{K})$ is surjective, 
The map $\iota$ lifts to $\iota^{sc}: \Gamma_K\to G_1^{sc}(\overline{K})$.

Consider the commutative diagrams of homomorphisms of $\overline{K}$-groups:
\begin{equation}\label{lem-inner-form-eq}
\begin{tikzcd}
{\overline{G_1^{sc}}} \arrow[r,"f^{sc}"] \arrow[d,"\pi_1"] & {\overline{G_2^{sc}}}  \arrow[d,"\pi_2"]  
&{\overline{G_1^{sc}}} \arrow[r,"h^{sc}"] \arrow[d,"\pi_1"]         & {\overline{G_1^{sc}}}  \arrow[d,"\pi_1"]  \\
 {\overline{G_1}} \arrow[r,"f"] & {\overline{G_2}} 
 &  {\overline{G_1}} \arrow[r,"h"] & {\overline{G_1}} ,
\end{tikzcd}
\end{equation}
where $\pi_i$ are universal covering maps.
The universal property of the universal covering states that
given any homomorphism $f$ (resp. $h$) in the diagram (\ref{lem-inner-form-eq}), there exists a unique homomorphism $f^{sc}$ (resp. $h^{sc}$) such that the diagram commutes.

It follows that $f=\theta$ (resp. $h=\rho_{_{\iota(\sigma)}}$) in diagram (\ref{lem-inner-form-eq}) lifts to a unique isomorphism $\theta^{sc}: \overline{G_1^{sc}}\to \overline{G_2^{sc}}$ (resp. $\rho^{sc}_{\iota(\sigma)}:\overline{G^{sc}_{1}}\to \overline{G^{sc}_{1}}$). Moreover, since $\pi_1\circ \rho^{sc}_{\iota(\sigma)}=\rho_{\iota(\sigma)}\circ\pi_1$, and
$$ \pi_1\circ (\theta^{sc})^{-1}\circ \sigma(\theta^{sc})=(\theta^{sc})^{-1} \circ \pi_2 \circ \sigma(\theta^{sc})
=(\theta^{sc})^{-1}\circ \sigma(\theta^{sc}) \circ \pi_1 ,$$
 we obtain   
$$\rho^{sc}_{\iota(\sigma)}=(\theta^{sc})^{-1}\circ \sigma(\theta^{sc}) $$
by the uniqueness of $h^{sc}$. It follows that $G_2^{sc}$ is an inner form of $G_1^{sc}$.
\end{proof}

\begin{proof}[Proof of Theorem \ref{thm-(ASA) with inner forms}] 
By Theorem \ref{thm-mainthm}, it suffices to show that $\hat{Q}$ in Theorem \ref{thm-mainthm} is split over $L$. 
As the kernel of a central isogeny, the group $Q$ is contained in the center of $Z(G^{red})^0\times G^{sc}$.
The inclusion $Q\subset Z(G^{red})^0\times Z(G^{sc})$ induces a surjective homomorphism of $\Gamma_K$-modules 
$$\widehat{Z(G^{red})^0}\oplus \widehat{Z(G^{sc})}\to \hat{Q} .$$
Since $Z(G^{red})^0$ is already split over $E\subset L$ (by hypothesis), 
it suffices to show that $\widehat{Z(G^{sc})} $ is split over $L$.

Now, the semi-simple group $G^{ss}$ is an inner form of a $K$-split group over $M$ (by hypothesis), 
hence $G^{sc}$ is also an inner form of a $K$-split group $G'$ over $M$ (Lemma \ref{lem-inner-form}), 
and we have $Z(G^{sc}_M)\cong Z(G'_M)$ (see the middle of page 517 in \cite[\S 24.c]{Mil17}).
Since $G'$ is split, its center $Z(G')$ is contained in a split maximal torus $\mathbb{G}_{m,L}^r$ for some $r$, 
and therefore $\widehat{Z(G')} $ is also split.
It follows that $\widehat{Z(G^{sc})} $ is split over $M\subset L$, which completes the proof.
\end{proof}

\section{The index of almost strong approximation}

Let $K$ be a number field and $G$ a connected linear algebraic group over $K$, and $S$ an infinite set of places.
If $G$ satisfies (ASA) off $S$, it is a natural to study the index $[G(\mathbb{A}^{S}_{K}):\overline{G(K)}^{S}]$,
which we call \textbf{the index of (ASA)}.

In this section, we aim to bound the index of (ASA) under the hypothesis of Theorem \ref{thm-mainthm}.
There are two approaches. 
The first is to follow the proof of Theorem \ref{thm-mainthm} 
and bound the cardinality of the Galois cohomology of $\hat{C_0}$ defined in (\ref{def-C0});
The second is to use the maximal torus $T$ and bound the cardinality of the Galois cohomology of $\hat{T}$.

The first approach involves computing the hyper-cohomology of $\hat{C_0}$, which is generally difficult. We therefore adopt the second approach, though it requires a stronger assumption that $T$ is split over $L$.

The following result explains the relationship between the $S$-Shafarevich group of $G$ and that of its maximal torus $T$.
Recall $\hat{C}=[\hat{T}\to \hat{T}^{sc}]$ in (\ref{natation-hatC}).

\begin{prp}\label{prop-relation-max-torus}
Let $G$ be a connected linear algebraic group over a number field $K$, and $T\subset G^{red}$ a maximal torus. 
If $T$ satisfies (ASA) off $S$, then $G$ satisfies (ASA) off $S$ and we have
\begin{equation}\label{compute-G-by-T}
|B_S(G)|\leq |B_S(T)| \cdot |H^1(K,\hat{T}^{sc}) |.
\end{equation}
\end{prp}

\begin{proof} 
The canonical distinguished triangle of $\hat{C}$ (see (\ref{disting-trian})) induces the following commutative diagram with exact rows:
$$
\begin{tikzcd}
{H^1(K,\hat{T}^{sc})} \arrow[r] \arrow[d]    & {H^1(K,\hat{C})} \arrow[r] \arrow[d]   & {H^2(K,\hat{T})}  \arrow[d]  \\
{\prod_{v\in S}H^1(K_{v},\hat{T}^{sc})} \arrow[r] & {\prod_{v\in S}H^1(K_{v},\hat{C})} \arrow[r] &
{\prod_{v\in S}H^2(K_{v},\hat{T})} .
\end{tikzcd}
$$
 We have $B_{S}(G)\cong \Sha^1_{S}(K,\hat{C})$ and $B_S(T)\cong \Sha^2_{S}(K,\hat{T})$ by (\ref{result-C0}). 
A diagram chasing shows that
$$\mathrm{Ker}(B_{S}(G)\to B_{S}(T) ) \subset \mathrm{Im}(H^1(K,\hat{T}^{sc})).$$
Since $H^1(K,\hat{T}^{sc})$ is finite, the group $B_{S}(G)$ is finite and we have (\ref{compute-G-by-T}).
Then Theorem \ref{thm-mainlem} implies that $G$ satisfies (ASA) off $S$.
\end{proof} 

\begin{coro}
Let $G$ be a connected linear algebraic group over a number field $K$ and $T\subset G^{red}$ a maximal torus.
Let $L$ be a splitting field of $T$, and assume that $\delta_L(S_{L})>0$. Then $G$ satisfies (ASA) off $S$.
\end{coro}

\begin{prp}\label{bound-torus-ss}
Under the hypotheses of Theorem \ref{thm-mainlem}, we have the following inequalities:

(1) if $G$ is a torus of rank $r$, then
$$[G(\mathbb{A}^{S}_{K}):\overline{G(K)}^{S}]\leq \delta_L(S_L)^{-r}\cdot |H^2(L/K,\hat{G}) |; $$

(2) if $G$ is semi-simple and $\mathrm{Pic}(\overline{G}) $ is generated by $r$ elements, then 
$$[G(\mathbb{A}^{S}_{K}):\overline{G(K)}^{S}]\leq \delta_L(S_L)^{-r}\cdot  | H^1(L/K, \mathrm{Pic}(\overline{G}))|. $$
\end{prp}

\begin{proof}
Let $E_{S_L}/L$ be the maximal abelian extension that splits over $S_L$.
 Lemma \ref{max-split-field-lem2} implies 
$$\Sha^1_{S_L}(L,\mathbb{Q}/\mathbb{Z}) \cong H^1(E_{S_L}/L,\mathbb{Q}/\mathbb{Z})\cong Hom(\mathrm{Gal}(E_{S_L}/L),\mathbb{Q}/\mathbb{Z}) .$$
By (\ref{density-extension}),
one has $|\Sha^1_{S_L}(L,\mathbb{Q}/\mathbb{Z}) |=[E_{S_L}:L] \leq \delta_L(S_L)^{-1}.$
Then Lemma \ref{sha-Z-Q/Z-Z/n} implies
$$ |\Sha^1_{S_L}(L,\mathbb{Z}/n) | \leq  |\Sha^2_{S_L}(L,\mathbb{Z}) |   \leq  \delta_L(S_L)^{-1} .$$
Recall the notation $\hat{C_0}$ defined in (\ref{def-C0}). By (\ref{Hochchild-counting}) and (\ref{result-C0}), we obtain
$$  |B_S(G)|\leq |\Sha^1_{S_{L}}(L,\hat{C_0}) | \cdot |H^2(L/K,H^{-1}(L,\hat{C_0})) | \cdot
| H^1(L/K,H^0(L,\hat{C_0}))|.  $$

If $G$ is a torus of rank $r$, then $\hat{C_0}=[\hat{G}\to 0]$ with $\hat{G}\cong \mathbb{Z}^r$.
Hence $H^0(L,\hat{C_0})=0$, and we have $|\Sha^1_{S_{L}}(L,\hat{C_0}) |\leq \delta_L(S_L)^{-r}$.

If $G$ is semi-simple and $\mathrm{Pic}(\overline{G}) $ is generated by $r$ elements, 
then $\hat{C_0}=[0\to \mathrm{Pic}(\overline{G}) ]$, and there is a surjective homomorphism $\mathbb{Z}^r\to \mathrm{Pic}(\overline{G})$.
It follows that $H^{-1}(L,\hat{C_0})=0$ and $|\Sha^1_{S_{L}}(L,\hat{C_0}) |\leq \delta_L(S_L)^{-r}$, which completes the proof of the Proposition.
\end{proof}

\begin{coro} \label{thm-bound by maximal torus}
Let $G$ be a connected linear algebraic group over a number field $K$ and $S\supset \infty_{K}$ an infinite set of places.
Let $T\subset G^{red}$ be a $r$-dimensional maximal torus with splitting field $L$ such that  $\delta_L(S_{L})>0$. 
Then 
\[
[G(\mathbb{A}^{S}_{K}):\overline{G(K)}^{S}]\leq \delta_L(S_{L})^{-r}\cdot |H^1(L/K,\hat{T}^{sc})|\cdot |H^2(L/K,\hat{T})|.
\]
\end{coro}

\begin{proof}
This is an immediate consequence of Proposition \ref{bound-torus-ss} (1) and Proposition \ref{prop-relation-max-torus}.
\end{proof}

\begin{exam} Assume that $\delta_{K}(S)>0$. 

(1) Let $G=\mathrm{GL}_{n}$. By Corollary \ref{thm-bound by maximal torus}, we have:
\[
|\mathrm{GL_{n}}(\mathbb{A}^{S}_{K}):\overline{\mathrm{GL}_{n}(K)}^{S}|\leq \delta_{K}(S)^{-n}.
\]

(2) Let $G=\mathrm{PGL}_{n}$. By Proposition \ref{bound-torus-ss} (2), we have 
\[
|\mathrm{PGL_{n}}(\mathbb{A}^{S}_{K}):\overline{\mathrm{PGL}_{n}(K)}^{S}|\leq \delta_{K}(S)^{-1}.
\]
In particular, $\mathrm{PGL}_{n}$ satisfies (SA) off $S$ provided that $\delta_{K}(S)>1/2$.

(3) Let $T=\mathrm{Res}_{L/K}\mathbb{G}_{m}$. 
In this case, the Shafarevich group can be computed directly, namely
\[
B_{S}(T)\cong \Sha^2_{S}(K,\hat{T})\cong \Sha^1_{S_{L}}(L,\mathbb{Q}/\mathbb{Z}).
\]
If $\delta_L(S_{L})>0 $, then 
\[
|T(\mathbb{A}^{S}_{K}):\overline{T(K)}^{S}|\leq \delta_L(S_{L})^{-1}.
\]
\end{exam}

\textbf{Acknowledgments.} The authors thank Zhizhong Huang, Ping Xi, Pengyu Yang and Yang Zhang for helpful discussions.

\bibliographystyle{alpha}
\bibliography{ref}
\end{document}